\documentclass[a4paper,12pt]{amsart}
\usepackage{amsthm, amsfonts, amssymb, mathtools, color, mathtools}
\usepackage[all]{xy}

\def\ZZ{\mathbb{Z}}
\def\QQ{\mathbb{Q}}

\def\CC{\mathbb{C}}

\def\HH{\mathbb{H}}
\DeclareMathOperator{\id}{id}
\DeclareMathOperator{\PSL}{PSL}
\DeclareMathOperator{\QMF}{\mathit{QM}}
\DeclareMathOperator{\MF}{\mathit{M}}
\DeclareMathOperator{\Hol}{\mathcal{O}}
\DeclareMathOperator{\Sol}{Sol}
\DeclareMathOperator{\Mat}{Mat}
\def\dl{\langle\kern-.22em\langle}
\def\dr{\rangle\kern-.22em\rangle}
\def\dal{\langle\kern-.22em\langle}
\def\dar{\rangle\kern-.22em\rangle}
\def\dbl{[\kern-.12em[}
\def\dbr{]\kern-.12em]}
\def\dpl{(\kern-.2em(}
\def\dpr{)\kern-.2em)}
\def\dast{\ast\kern-.03em\ast}
\def\tast{\ast\kern-.26em\ast\kern-.26em\ast}
\title[Extremal quasimodular forms]{Extremal quasimodular forms of lower depth with integral Fourier coefficients}
\author{Tsudoi Kaminaka and Fumiharu Kato}
\begin{document}
\pagestyle{myheadings}
\theoremstyle{plain}
\newtheorem{thm}[subsection]{Theorem}
\newtheorem{prop}[subsection]{Proposition}
\newtheorem{lem}[subsection]{Lemma}
\newtheorem{cor}[subsection]{Corollary}
\newtheorem{probs}[subsection]{Problems}
\newtheorem{cla}{Claim}[subsection]
\theoremstyle{definition}
\newtheorem{dfn}[subsection]{Definition}
\newtheorem{ntn}[subsection]{Notation}
\newtheorem{exa}[subsection]{Example}
\newtheorem{exas}[subsection]{Examples}
\newtheorem{assum}[subsection]{Assumption}
\newtheorem{sit}[subsection]{Situation}
\theoremstyle{remark}
\newtheorem{rem}[subsection]{Remark}
\newtheorem{note}[subsection]{Note}
%
\makeatletter
  \renewcommand{\theequation}{%
   \arabic{equation}}
  \@addtoreset{equation}{section}
 \makeatother
\def\theequation{\thesection.\arabic{equation}}
\maketitle

\begin{abstract}
We show that, based on Grabner's recent results on modular differential equations satisfied by quasimodular forms, there exist only finitely many normalized extremal quasimodular forms of depth $r$ that have all Fourier coefficients integral for each of $r=1,2,3,4$, and partly classifies them, where the classification is complete for $r=2,3,4$; in fact, we show that there exists no normalized extremal quasimodular forms of depth $4$ with all Fourier coefficients integral.
Our result disproves a conjecture by Pellarin.
\end{abstract}

\setcounter{section}{0}
\section{Introduction}\label{sec-intro}
The notion of Quasimodular forms has been first introduced by M.\ Kaneko and D.\ Zagier in the middle of 1990's (\cite{KZ}). 
Since then, the connections with various fields have been gradually clarified, and accordingly, many people have been paying attention to this subject.
Especially rich is the relation with linear differential equations, which was already pointed out in early papers by M.\ Kaneko and M.\ Koike (\cite{KK2}\cite{KK3}), and has been at one of the centers of interest in this fertile field of mathematics.
Not a few people have entered into this research field and made a lot of progress.
See, for example, the reference list in \cite{Gr} for these previous works.

Our research in this paper owes much to P.J.\ Grabner's recent work \cite{Gr}, in which he has carried out an impressive and extensive study of the so-called ``modular differential equations'', which allows him to deduce a lot of characterization and existence results on certain type of quasimodular forms, the so-called {\em balanced quasimodular forms}.
Applying these results to a more specific type of quasimodular forms, the {\em extremal quasimodular forms}, which has been introduced by Kaneko-Koike \cite{KK1}, Grabner could give a very concrete description of them.

In this paper, we aim to obtain even more concrete results for extremal quasimodular forms of degree at most $4$, based on Grabner's results.
Note that, for $r=1,2,3,4$ and $w\geq 2r$ with $w-2\neq 2r$, the normalized extremal quasimodular form of weight $w$ and depth $r$ is known to exist, and is unique (\cite{PN}).
Our interest lies in when the Fourier coefficients of these extremal quasimodular forms are integers, and if so, are they positive or not. 
In connection with this, Pellarin \cite{PN} has conjectured that, for each depth $r=1,2,3,4$, the extremal quasimodular forms of depth $r$ have integral Fourier coefficients (i.e, all coefficients are integers) for infinitely many weights (\cite[Conjecture 3.9]{PN}).
Our main theorem not only disproves this conjecture, but also partly determine the weights and depths for which the normalized extremal quasimodular forms have all coefficients integral. 
To state it, let us say that a quasimodular form $f$ has {\em integral $q$-expansion}, if the $q$-expansion $f(q)$ belongs to $\ZZ\dbl q\dbr$.

\begin{thm}\label{thm-main}
For $r=1,2,3,4$, let $\mathcal{E}_r$ be the set of weights $w$ such that the normalized extremal quasimodular form $f^{(r)}_w$ of weight $w$ and depth $r$ has integral $q$-expansion.
Then the sets $\mathcal{E}_r$ are finite for $r=1,2,3,4$.
Moreover, we have
\begin{gather}
\mathcal{E}_1\subset\{2,6,8,10,12,14,16,18,20,22,24,28,30,32,34,38,54,58,68,80,114,118\}\\
\mathcal{E}_2=\{4,8\}\\
\mathcal{E}_3=\{6\}\\
\mathcal{E}_4=\emptyset
\end{gather}
\end{thm}

\begin{rem}\label{rem-main}
We could not completely determine the set $\mathcal{E}_1$. 
It is (almost) trivial that $2,6,8,10\in\mathcal{E}_1$.
In the appendix of \cite{PN}, G.\ Nebe showed by using theta series associated to the Leech lattice that $14\in\mathcal{E}_1$.
Furthermore, P.J.\ Grabner informed the second author in private communication that the fact $14\in\mathcal{E}_1$ follows also by applying \cite[(12.7)]{BO} (and the first equality of \cite[(12.4)]{BO}) to the expression
$$
f^{(1)}_{14}=\frac{1}{4146}\sum^{\infty}_{n=1}(n\sigma_{11}(n)-n\tau(n))q^n,
$$
and similarly that, due to the third equality of \cite[(12.6)]{BO} applied to 
$$
f^{(1)}_{12}=\frac{1}{1050}\sum^{\infty}_{n=1}(n\sigma_9(n)-\tau(n))q^n,\\
$$
we have $12\in\mathcal{E}_1$.
We thank the referee to point out that the last fact implies $16\in\mathcal{E}_1$ (cf.\ \S\ref{sub-d1w4}).
In sum, what we know so far is that $2,6,8,10,12,14$ and $16$ belong to $\mathcal{E}_1$.
\end{rem}

\begin{rem}\label{rem-main2}
In his recent paper \cite{Gr2}, Grabner proved that all but finitely many $q$-coefficients of normalized extremal quasimodular forms of depth $\leq 4$ are positive, which partly solves the conjecture by Kaneko-Koike \cite{KK1}.
In \S\ref{sec-cases}, as a byproduct of our argument, we will show that, at least for $(w,r)=(6,3), (8,2)$, the $q$-coefficients of the normalized extremal quasimodular form of weight $w$ and depth $r$ are all positive.
In fact, our method of proving positivity can treat other cases, which, however, we are not going to discuss in this paper.
\end{rem}

The composition of this paper is as follows.
In the next section, we will collect some basics on the modular differential equations.
Our treatment of them will be given from the viewpoint of differential operators, whence introducing {\em modular differential operators}.
The explicit calculation of Fourier coefficients will become smoother by matrix interpretation of these operators.
All of this will be the topic of this section.
In \S\ref{sec-eqmf}, we will briefly describe the basics on extremal quasimodular forms, especially in relation with MDO's, following Grabner's result in \cite{Gr}.
Then, from \S\ref{sec-d1} onward, we will devote ourselves to the concrete calculations.
In the first four sections (\S\ref{sec-d1} $\sim$ \S\ref{sec-d4}), we will carry out explicit calculation of the first $q$-coefficients of the extremal quasimodular forms of depth at most $4$ in order to narrow down the possibility of the weights that falls in $\mathcal{E}_r$, and in the last section, we show some concrete examples that actually admit integral $q$-expansion.
The proof of the main theorem (Theorem \ref{thm-main}) will be done by combination of these concrete calculations.

The authors thank Masanobu Kaneko and Yuichi Sakai in Kyushu University for valuable comments and encouragements.
The second author thanks Peter Grabner, Federico Pellarin and Gabriele Nebe for valuable comments on the first draft of this paper.
Thanks are also due to the referee for careful read and several valuable comments, which simplified arguments in \S\ref{sub-d1w4} and \S\ref{sec-cases}

\subsection{Convention}
\begin{itemize}
\item Modular forms and quasimodular forms throughout this paper are those on $\PSL(2,\ZZ)$.
\item $\sigma_k(n)$ denotes the $k$-th divisor function, i.e., 
$$
\sigma_k(n)=\sum_{d|n}d^k.
$$
\item $E_2,E_4,E_6,\ldots$ will denote the standard Eisenstein series; i.e., 
$$
E_{2k}=1-\frac{4k}{B_{2k}}\sum^{\infty}_{n=1}\sigma_{2k-1}(n)q^n,
$$
where $B_k$ is the $k$-th Bernoulli number.
\end{itemize}

\section{Modular differential equations}\label{sec-MDE}
\begin{dfn}\label{dfn-MDO}
A {\em modular differential operator} (abbr.\ MDO), as defined in \cite[\S2]{Gr}, is a linear differential operator on the space $\Hol(\HH)$ of holomorphic functions on the upper-half plane $\HH$ of the form
\begin{equation}\label{eqn-MDE}
D=B_m\partial^{r+1}_{w-r}+B_{m+2}\partial^r_{w-r}+\cdots+B_{m+2r}\partial_{w-r}+B_{m+2r+2},
\end{equation}
where 
\begin{itemize}
\item[(a)] $B_m,B_{m+2},\ldots,B_{m+2r+2}$ are modular forms of weights $m,m+2,\ldots,m+2r+2$, respectively, and $B_m(i\infty)=1$,
\item[(b)] $\partial^r_w$ are the {\em iterated Serre differentials}, which are defined recursively by
\begin{equation}
\partial^0_w=\id,\quad \partial_w=\frac{1}{2\pi i}\frac{d}{dz}-\frac{w}{12}E_2
,\quad \partial^{r+1}_w=\partial_{w+2r}\partial^r_w.
\end{equation}
\end{itemize}
The MDO (\ref{eqn-MDE}) is said to be {\em normalized} if $m=0$ and $B_m=1$.
\end{dfn}

If $D$ is an MDO, then the solution space $\Sol(D)=\{f\in\Hol(\HH)\mid Df=0\}$ of $D$ is naturally acted on by $\PSL(2,\ZZ)$ (\cite[Lem.\ 2.4]{Gr}).
Any $T$-invariant solution $f$ of $D$ admits Fourier expansion of the form 
\begin{equation}\label{eqn-holosol}
f=q^{\lambda}\sum^{\infty}_{n=0}a(n)q^n, \quad a(0)\neq 0
\end{equation}
where $q=e^{2\pi iz}$ and $\lambda$ is a non-negative integer.
Conversely, any holomorphic solution as in (\ref{eqn-holosol}) of $D$ considered around the cusp $i\infty$ extends uniquely to a $T$-invariant holomorphic solution on $\HH$.
Note that the Serre differential $\partial_w$ with respect to the $q$-coordinate is given by
\begin{equation}
\partial_w=q\frac{d}{dq}-\frac{w}{12}E_2.
\end{equation}

\subsection{Matrix representation}\label{sub-matrix}
We denote by $\Mat(m,n)$ the space of all $m\times n$ matrices $A=[a_{ij}]$ over $\CC$, whose rows (resp.\ columns) are numbered as $i=0,\ldots,m-1$ (resp.\ $j=0,\ldots,n-1)$, and write $\Mat(n)=\Mat(n,n)$.
Consider for any $n\leq m$ the map $\pi_{nm}\colon\Mat(m)\rightarrow\Mat(n)$ that maps each $A\in\Mat(m)$ to its upper-left $m\times m$ block, and let $\Mat(\infty)=\varprojlim_n\Mat(n)$ together with the projections $\pi_n\colon\Mat(\infty)\rightarrow\Mat(n)$ be the projective limit of the projective system $\{\Mat(n),\pi_{nm}\}$ thus obtained.

Let $D$ be an MDO. Then for any $\lambda,n\in\ZZ_{\geq 0}$, $D$ induces a $\CC$-limear selfmap on $q^{\lambda}\CC[q]/q^{\lambda+n}\CC[q]$.
We denote the matrix representation of this map with respect to the basis $\{q^{\lambda},q^{\lambda+1},\ldots,q^{\lambda+n-1}\}$ by $D(\lambda;n)$.
The matrix $D(\lambda,n)$ is an element of $\Mat(n)$, and we have $\pi_{nm}(D(\lambda;m))=D(\lambda;n)$ for any $n\leq m$.
Hence one has a unique $D(\lambda)\in\Mat(\infty)$ such that $\pi_n(D(\lambda))=D(\lambda;n)$ for any $n$.
We call $D(\lambda)$ the {\em matrix representation} of $D$.

\begin{exa}\label{exa-matrix1}
Let $B=\sum^{\infty}_{n=0}b(n)q^n$ be a modular form.
Then the $(i,j)$-entry of $B(\lambda)$ is given as follows:
$$
B(\lambda)_{ij}=\begin{cases}0&(i<j),\\ b(i-j)&(i\geq j).\end{cases}
$$
\end{exa}

\begin{exa}\label{exa-matrix2}
The matrix representation of the Serre differential $\partial_w$ is given as follows:
$$
\partial_w(\lambda)_{ij}=\begin{cases}0&(i<j),\\ \lambda+i-\frac{w}{12}&(i=j),\\ 2w\cdot\sigma_1(i-j)&(i>j).\end{cases}
$$
\end{exa}

\subsection{Indicial equation}\label{sub-indicial}
The $(0,0)$-entry of $D(\lambda)$ is the degree-$\lambda$ coefficient of $Dq^{\lambda}$.
If $D$ is given as in (\ref{eqn-MDE}), then it is a polynomial of $\lambda$ given by
\begin{equation}\label{eqn-indicial1}
p_D(\lambda)=\sum^{r+1}_{k=0}B_{m+2k}(i\infty)q_{r+1-k}(\lambda,w),
\end{equation}
where
\begin{equation}\label{eqn-indicial2}
q_l(x,w)=\bigg(x-\frac{w-r}{12}\bigg)\bigg(x-\frac{w-r+2}{12}\bigg)\cdots\bigg(x-\frac{w-r+2l-2}{12}\bigg).
\end{equation}
The equation $p_D(\lambda)=0$ is the {\em indicial equation}, whose roots are {\em characteristic exponents}.
For any $n\in\ZZ_{\geq 0}$, the $(n,n)$-entry of $D(\lambda)$ is equal to $p_D(\lambda+n)$.

\subsection{Solution by matrix entries}\label{sub-solutionmatrix}
Suppose that $\lambda=\lambda_0$ is a simple root of $p_D(\lambda)=0$, and that $\lambda=\lambda_0+n$ for any $n\in\ZZ_{>0}$ is not a root, then by Frobenius method we obtain the unique normalized solution of $D$ 
$$
f=q^{\lambda_0}\sum^{\infty}_{n=0}a(n)q^n,\quad a(0)=1
$$
with
\begin{equation}\label{eqn-coeffsol}
a(n)=\sum_{0=i_0<i_1<\cdots<i_s=n}(-1)^{s+1}\prod^s_{k=0}\frac{D(\lambda)_{i_{k+1},i_k}}{D(\lambda)_{i_{k+1},i_{k+1}}}
\end{equation}
for $n\in\ZZ_{>0}$, where the sum is taken over all sequences of integers of the form $0=i_0<i_1<\cdots<i_s<i_{s+1}=n$.

\section{Extremal quasimodular forms}\label{sec-eqmf}
\subsection{Quasimodular forms}\label{sub-qmf}
We denote by $\MF_w$ the $\CC$-vector space of all modular forms of weight $w$.
The $\CC$-vector space of {\em quasimodular forms} (due to M.\ Kaneko and D.\ Zagier \cite{KZ}) of weight $w$ and depth $\leq r$ is 
$$
\QMF^r_w=\bigoplus^r_{k=0}E^k_2M_{w-2k}.
$$

It follows from the definition that any quasimodular form is invariant under $T$, and admits Fourier expansion.

Quasimodular forms of depth at most $4$ have several special features, as manifested in several places in \cite{Gr}, one of which is the following simple dimension formula:
For $r\leq 4$ and $w(r+1)\equiv 0\pmod{12}$, 
$$
\dim\QMF^r_w=\bigg\lfloor\frac{w(r+1)}{12}\bigg\rfloor+\begin{cases}0&(\textrm{$r=4$ and $w\equiv 10$ (mod $12$)})\\ 1&(\textrm{otherwise})\end{cases}
$$

Almost throughout this paper, we will focus on the so-called {\em extremal quasimodular forms} (due to M.\ Kaneko and M.\ Koike \cite{KK1}) of depth at most $4$.
\begin{dfn}[{\rm Kaneko-Koike \cite{KK1}}]\label{dfn-eqmf}
Let $f$ be a quasimodular form of weight $w$ and depth $r$, i.e., an element of $\QMF^r_w\setminus\QMF^{r-1}_w$. Then $f$ is said to be {\em extremal} if its Fourier expansion is of the form 
\begin{equation}\label{eqn-eqmf}
f=q^{\lambda}\sum^{\infty}_{n=0}a(n)q^n,\quad a(0)\neq 0,
\end{equation}
where $\lambda=\dim\QMF^r_w-1$.
If moreover $a(0)=1$, $f$ is further said to be {\em normalized}.
\end{dfn}

It was conjectured in \cite{KK1} that extremal quasimodular forms should exist (and should be unique when normalized) for pairs $(w,r)$ with naturally required numerical constraints.
Grabner \cite{Gr} has conducted an extensive study of quasimodular forms as solutions of modular differential equations, and has obtained crucial results on the so-called {\em balanced quasimodular forms}, which generalizes the notion of extremal quasimodular forms.
Let us present here some of his results restricted on extremal quasimodular forms.

\begin{thm}[{\rm cf.\ \cite[4.8 \& 4.11]{Gr}}]\label{thm-grabner1}
Let $r$ be a positive integer at most $4$ and $w$ a positive even integer such that $w(r+1)\equiv 0\pmod{12}$.

$(1)$ Every extremal quasimodular form of depth $r$ and weight $w$ is a solution of a normalized MDO of the form
\begin{equation}\label{eqn-MDOeqmf}
\begin{split}
D=\partial^{r+1}_{w-r}+a_4E_4\partial^{r-1}_{w-r}+a_6E_6\partial^{r-2}_{w-r}+\cdots+a_{2r}&E_{2r}\partial_{w-r}+a_{2r+2}E_{2r+2}\\
&(a_4,a_6,\ldots,a_{2r+2}\in\QQ).
\end{split}
\end{equation}

$(2)$ Conversely, if the indicial equation of an MDO as in {\rm (\ref{eqn-MDOeqmf})} is of the form $p_D(x)=x^r(x-\lambda)$, where $\lambda$ is a positive integer, then the solution of $D$ of the form as in {\rm (\ref{eqn-eqmf})} $($which is unique up to non-zero factor$)$ is an extremal quasimodular form of weight $w$ and depth $r$.
\end{thm}

The significance of the last theorem lies in that it paves the way to give explicit displays of (the Fourier expansion of) extremal quasimodular forms of depth lower than or equal to $4$, especially combined with the matrix calculation as in \ref{sub-solutionmatrix}.

\section{Depth one}\label{sec-d1}
If the depth $r=1$, then $\lambda$ in (\ref{eqn-eqmf}) is given by $\lambda=\lfloor\frac{w}{6}\rfloor$.
Hence one has three cases: $(w,\lambda)=(6k,k),(6k+2,k),(6k+4,k)$.

\subsection{Weight {\boldmath $\equiv 0\pmod 6$}}\label{sub-d1w0}
Due to Theorem \ref{thm-grabner1}, the normalized extremal quasimodular form $f_w$ of depth $1$ and weight $w=6k$ ($k>0$) is the solution of the MDO
$$
D_1=\partial^2_{w-1}-\frac{w^2-1}{12^2}E_4.
$$
The unique positive characteristic exponent is equal to $k$.
We look at the $3\times 3$ matrix $D_1(k;3)$ calculated as
$$
D_1(k;3)=\begin{bmatrix}0&0&0\\ 12k(1-4k)&k+1&0\\ 72k(1-5k)&12k(3-4k)&2(k+2)\end{bmatrix}.
$$
It follows from \ref{sub-solutionmatrix} that the first three terms of the Fourier expansion of $f_w$ are given by $f_{6k}=q^k\sum^{\infty}_{n=0}a(n)q^n$, where
\begin{equation}\label{eqn-d1w0}
\begin{split}
&a(0)=1,\quad a(1)=48k-\frac{60k}{k+1},\\
&a(2)=36(32k^2-123k+315)-2520\frac{10k+9}{(k+1)(k+2)}.
\end{split}\end{equation}
We consider the condition for these numbers to be integers.
First, since $k$ and $k+1$ are mutually prime, the value of $a(1)$ already restricts $k$ into a finite set of integers, viz., those $k$ such that $k+1$ divides $60$.
We have thus 11 candidates $k=1,2,3,4,5,9,11,14,19,29,59$.
We next look at $a(2)$, especially its fractional part $2520(10k+9)/\{(k+1)(k+2)\}$, and dismiss $k=11,14,29,59$.
So $k$ must be one of the $7$ numbers $1,2,3,4,5,9,19$.

\subsection{Weight {\boldmath $\equiv 2\pmod 6$}}\label{sub-d1w2}
The normalized extremal quasimodular form $f_{w+2}$ of depth $1$ and weight $w+2=6k+2$ is $E_2$ if $k=0$, and for $k>0$, given by 
$$
f_{6k+2}=\frac{12}{6k+1}\partial_{6k-1}f_{6k}
$$
due to \cite[\S6.1]{Gr}.
We calculate, by means of Example \ref{exa-matrix2} and (\ref{eqn-d1w0}), the first three Fourier coefficients of $f_{6k+2}=q^k\sum^{\infty}_{n=0}b(n)q^n$ as
\begin{equation}\label{eqn-d1w2}
\begin{split}
&b(0)=1,\quad b(1)=48k+60-\frac{84}{k+1},\\
&b(2)=36(32k^2+37k-247)+2520\frac{10k+7}{(k+1)(k+2)}.
\end{split}\end{equation}
Similarly to the previous case, first by $b(1)$, we find that $k$ should be among the 11 numbers $1,2,3,5,6,11,13,20,27,41,83$, and then by $b(2)$, we dismiss $20,27,41,83$.
So in this case $k$ must be one of the $7$ numbers $1,2,3,5,6,11,13$.

\subsection{Weight {\boldmath $\equiv 4\pmod 6$}}\label{sub-d1w4}
Again due to \cite[\S6.1]{Gr} the normalized extremal quasimodular form $f_{w+4}$ of depth $1$ and weight $w+4=6k+4$ ($k>0$) is given by 
$$
f_{6k+4}=E_4f_{6k}.
$$
It immediately follows from this that $f_{6k+4}$ has integral $q$-expansion if and only if so does $f_{6k}$, since $E_4$, and hence $E^{-1}_4$ as well, belongs to $1+q\ZZ\dbl q\dbr$.
Thus, by \S\ref{sub-d1w0}, $k$ must be one of the $7$ numbers $1,2,3,4,5,9,19$.

\section{Depth two}\label{sec-d2}
If the depth $r=2$, then $\lambda$ in (\ref{eqn-eqmf}) is given by $\lambda=\lfloor\frac{w}{4}\rfloor$.
Hence one has two cases: $(w,\lambda)=(4k,k),(4k+2,k)$.

\subsection{Weight {\boldmath $\equiv 0\pmod 4$}}\label{sub-d2w0}
The normalized extremal quasimodular form $f_w$ of depth $2$ and weight $w=4k$ $(k>0)$ is the solution of the MDO
$$
D_2=\partial^3_{w-2}-\frac{3w^2-4}{12^2}E_4\partial_{w-2}-\frac{(w+1)(w-2)^2}{6\cdot 12^2}E_6,
$$
which has the unique positive characteristic exponent $k$. 
In this case, it is actually enough to look at the $2\times 2$ block
$$
D_2(k;2)=\begin{bmatrix}0&0\\ -8k(k^2+3k-1)&(k+1)^2\end{bmatrix}.
$$
Thus, the first two terms of $f_{4k}=q^k\sum^{\infty}_{n=0}a(n)q^n$ are given by
\begin{equation}\label{eqn-d2w0}
a(0)=1,\quad a(1)=8k-\frac{8k(k-2)}{(k+1)^2}.
\end{equation}
Since $k$ and $k+1$ are mutually prime, for $a(1)$ to be an integer it is necessary that $(k+1)^2$ divides $8(k-2)$.
This is possible if and only if $k=1,2$.

\subsection{Weight {\boldmath $\equiv 2\pmod 4$}}\label{sub-d2w2}
Due to \cite[\S6.2]{Gr} the normalized extremal quasimodular form $f_{w+2}$ of depth $2$ and weight $w+2=4k+2$ ($k>0$) is given by 
$$
f_{4k+2}=\frac{6}{4k+1}\partial_{4k-2}f_{4k}.
$$
We calculate, by means of Example \ref{exa-matrix2} and (\ref{eqn-d2w0}), the first two Fourier coefficients of $f_{4k+2}=q^k\sum^{\infty}_{n=0}b(n)q^n$ as
\begin{equation}\label{eqn-d2w2}
b(0)=1,\quad b(1)=8k+32-\frac{8(4k+7)}{(k+1)^2}.
\end{equation}
It is easy to see that $b(1)$ is an integer if and only if $k=1$, but this doesn't happen, since there exists no quasimodular form of weight $6$ and depth $2$ (cf.\ \cite[p.458]{KK1}).

\section{Depth three}\label{sec-d3}
If the depth $r=3$, then $\lambda$ in (\ref{eqn-eqmf}) is given by $\lambda=\lfloor\frac{w}{3}\rfloor$.
Hence one has three cases: $(w,\lambda)=(6k,2k),(6k+2,2k),(6k+4,2k+1)$.

\subsection{Weight {\boldmath $\equiv 0\pmod 6$}}\label{sub-d3w0}
Due to Theorem \ref{thm-grabner1}, the normalized extremal quasimodular form $f_w$ of depth $3$ and weight $w=6k$ ($k>0$) is the solution of the MDO
$$
D_3=\partial^4_{w-3}-\frac{3w^2-5}{72}E_4\partial^2_{w-3}-\frac{w^3-3w^2+5}{216}E_6\partial_{w-3}-\frac{(w+1)(w-3)^3}{6912}E^2_4
$$
The unique positive characteristic exponent is equal to $2k$.
We calculate $D_3(2k;3)$ as
$$
D_3(2k;3)=\begin{bmatrix}0&0&0\\ -12k(2k+1)(2k^2+5k-1)&(2k+1)^3&0\\ 288k(4k^3+6k^2-7k+1)&-12k(4k^3+36k^2+27k-15)&16(k+1)^3\end{bmatrix}.
$$
It follows from \ref{sub-solutionmatrix} that the first three terms of the Fourier expansion of $f_w$ are given by $f_{6k}=q^{2k}\sum^{\infty}_{n=0}a(n)q^n$, where
\begin{equation}\label{eqn-d3w0}
\begin{split}
&a(0)=1,\quad a(1)=k\bigg(6+\frac{18(2k-1)}{(2k+1)^2}\bigg),\\
&a(2)=k\bigg(18k+63-\frac{27(4k^4+48k^3+71k^2+10k+3)}{(k+1)^3(2k+1)^2}\bigg).
\end{split}\end{equation}
We claim that $a(1)\in\ZZ$ implies $k=1$.
Indeed, since $k$ and $2k+1$ are mutually prime, $(2k+1)^2$ has to divide $18(2k-1)$.
Since $(2k+1)^2>18(2k-1)$ for $k>7$, one only have to check $k=1,2,\ldots,7$, and thus we find the $k=1$ is the only possibility.

\subsection{Weight {\boldmath $\equiv 2\pmod 6$}}\label{sub-d3w2}
Due to \cite[\S6.3]{Gr} the normalized extremal quasimodular form $f_{w+2}$ of depth $3$ and weight $w+2=6k+2$ ($k>0$) is given by 
$$
f_{6k+2}=\frac{4}{6k+1}\partial_{6k-3}f_{6k}.
$$
We calculate, by means of Example \ref{exa-matrix2} and (\ref{eqn-d3w0}), the first two Fourier coefficients of $f_{6k+2}=q^{2k}\sum^{\infty}_{n=0}b(n)q^n$ as
\begin{equation}\label{eqn-d3w2}
\begin{split}
&b(0)=1,\quad b(1)=6k+21-\frac{9(6k+5)}{(2k+1)^2},\\
&b(2)=9(k+1)(2k+13)-\frac{27(64k^4+190k^3+159k^2+28k+7)}{(2k+1)^2(k+1)^3}.
\end{split}\end{equation}
One sees easily that $b(1)$ is an integer only if $k=1$; indeed, since $(2k+1)^2>9(6k+5)$ for $k>13$, one only have to check $k=1,2,\ldots,13$.
But there exists no quasimodular form of weight $8$ and depth $3$.

\subsection{Weight {\boldmath $\equiv 4\pmod 6$}}\label{sub-d3w4}
Again due to \cite[\S6.3]{Gr} the normalized extremal quasimodular form $f_{w+4}$ of depth $3$ and weight $w+4=6k+4$ ($k>0$) is given by 
$$
f_{6k+4}=\frac{2(6k+3)^2}{27(6k+1)(6k+2)^3}\bigg(\frac{(6k+1)(18k+1)}{48}E_4-\partial^2_{w-3}\bigg)f_{6k}
$$
From this, we can calculate the first two Fourier coefficients of $f_{6k+4}=q^{2k+1}\sum^{\infty}_{n=0}c(n)q^n$ as
\begin{equation}\label{eqn-d3w4}
c(0)=1,\quad c(1)=6(k+3)-\frac{3(3k+2)(3k+4)}{2(k+1)^3}.
\end{equation}
Since $2(k+1)^3>3(3k+2)(3k+4)$ for $k>12$, we check the cases $k=1,2,\ldots,12$, and find that $c(1)$ can never be an integer.

\section{Depth four}\label{sec-d4}
If the depth $r=4$, then one has six cases: $(w,\lambda)=(12k,5k),(12k+2,5k),(12k+4,5k+1),(12k+6,5k+2),(12k+8,5k+3),(12k+10,5k+3)$.

\subsection{Weight {\boldmath $\equiv 0\pmod{12}$}}\label{sub-d4w0}
The normalized extremal quasimodular form $f_w$ of depth $4$ and weight $w=12k$ ($k>0$) is the solution of the MDO
\begin{equation*}
\begin{split}
D_5=\partial^5_{w-4}&-\frac{5}{72}(w^2-2)E_4\partial^3_{w-4}-\frac{5}{432}(w^3-3w^2+6)E_6\partial^2_{w-4}\\ &-\frac{15w^4-120w^3+280w^2-496}{20736}E^2_4\partial_{w-4}-\frac{(w-4)^4(w+1)}{62208}E_4E_6
\end{split}
\end{equation*}
The unique positive characteristic exponent is equal to $5k$.
We calculate $D_5(5k;5)$ as
{\scriptsize 
\begin{equation*}
\begin{split}
&D_5(5k;5)=\\ 
&\left[\begin{matrix}0&0\\ -24k(211k^4+370k^3+90k^2-1)&(5k+1)^4\\ 72k(1349k^4+1780k^3-40k^2-200k+16)&-24k(211k^4+1110k^3+750k^2+60k-31)\\ -96k(4291k^4-2130k^3-4410k^2+1350k-81)&72k(1349k^4+3560k^3+50k^2-1240k+121)\\ -168k(8491k^4+20920k^3-22560k^2+4800k-256)&-96k(4291k^4-3550k^3-11730k^2+4850k-341)\end{matrix}\right.\\
&\left.\begin{matrix}
0&0&0\\ 0&0&0\\ 2(5k+2)^4&0&0\\ -24k(211k^4+1850k^3+2070k^2+300k-211)&3(5k+3)^4&0\\ 72k(1349k^4+5340k^3+200k^2-3960k+496)&-24k(211k^4+2590k^3+4050k^2+840k-781)&4(5k+4)^4\end{matrix}\right]
\end{split}
\end{equation*}}
It follows from \ref{sub-solutionmatrix} that the first three terms of the Fourier expansion of $f_w$ are given by $f_{6k}=q^{5k}\sum^{\infty}_{n=0}a(n)q^n$, where
\begin{equation}\label{eqn-d5w0}
\begin{split}
a(0)=&1\\
a(1)=&k\bigg(8+\frac{64k^4 + 4880k^3 + 960k^2 - 160k - 32}{(5 k + 1)^{4}}\bigg)\\
a(2)=&36k\big(356168 k^{9} + 1655115 k^{8} + 2916520 k^{7} + 2053130 k^{6} + 514604 k^{5} + 1611 k^{4} \\ &- 7300 k^{3} + 1160 k^{2} + 128 k - 16\big)\frac{1}{ (5 k + 1)^{4}(5 k + 2)^{4}}\\
a(3)=&32k\big(676363032 k^{14} + 5871071835 k^{13} + 22218453445 k^{12} + 45563807970 k^{11} \\ &+ 52449490244 k^{10} + 32410195422 k^{9} + 9395505420 k^{8} + 1068698970 k^{7} \\ &+ 405209936 k^{6} + 205193691 k^{5} + 13155691 k^{4} - 4967520 k^{3} - 219672 k^{2} \\ &+ 47952 k - 1296\big) \frac{1}{(5 k + 1)^{4}  (5 k + 2)^{4}  (5 k + 3)^{4}}\\ 
a(4)=&6k \big(4566803192064 k^{19} + 63266677462080 k^{18} + 401294985696140 k^{17} \\ &+ 1503115744273725 k^{16} + 3613880784409904 k^{15} + 5700525954443508 k^{14} \\ &+ 5816263091692920 k^{13} + 3712529153286830 k^{12} + 1502426035274784 k^{11} \\ &+ 548595090655756 k^{10} + 271944869947516 k^{9} + 85717030521645 k^{8} \\ &- 1106326811376 k^{7} - 4903195968296 k^{6} + 73922086048 k^{5} \\ &+ 175610335952 k^{4} - 2231627136 k^{3} - 1787539968 k^{2} \\ &+ 102629376 k - 2322432\big) \frac{1}{(5 k + 1)^{4}  (5 k + 2)^{4}  (5 k + 3)^{4}  (5 k + 4)^{4}}
\end{split}
\end{equation}
We need to check that there exists no positive integer $k$ such that $a(1)\in\ZZ$.
Since $k$ and $5k+1$ are mutually prime, it suffices to show that $(64k^4 + 4880k^3 + 960k^2 - 160k - 32)/(5k+1)^4$ can never be an integer.
Since the denominator exceeds the numerator for $k>7$, we check the cases $k=1,2,\ldots,7$, and find that none of them makes the fraction integral.

\subsection{Other weights}\label{sub-d4other}
Using (\ref{eqn-d5w0}) and \cite[Prop.\ 6.4]{Gr} one can compute first terms of the normalized extremal quasimodular forms $f_{w+2},f_{w+4},f_{w+6},f_{w+8},f_{w+10}$ of weights $w+2,w+4,w+6,w+8,w+10$, respectively.
The second Fourier coefficients of them (the weight and the degree are denoted respectively in the superscripts and subscripts) are calculated as follows:
\begin{equation}
\begin{split}
a^{w+2}_{5k+1}&=\frac{24(211k^5+579k^4+238k^3+6k^2-9k-1)}{(5k+1)^4}\\
a^{w+4}_{5k+2}&=\frac{24(211k^5+777k^4+784k^3+328k^2+60k+4)}{(5k+2)^4}\\
a^{w+6}_{5k+3}&=\frac{24(211k^5+903k^4+1286k^3+822k^2+243k+27)}{(5k+3)^4}\\
a^{w+8}_{5k+4}&=\frac{24(211k^5+1101k^4+2032k^3+1744k^2+712k+112)}{(5k+4)^4}\\
a^{w+10}_{5k+4}&=\frac{24(211k^5+1310k^4+2720k^3+2560k^2+1124k+186)}{(5k+4)^4}\\
\end{split}
\end{equation}

We need to show that these values are not integers for all positive integers $k$.
To check the first one, we consider
$$
5^4a^{w+2}_{5k+1}=24\cdot 211k+9845-\frac{125k^{4} +2112100k^{3} + 1488030k^{2} + 336964k + 24845}{(5k+1)^4},
$$
where the last fraction is smaller than $1$ for $k>4223$.
So we only have to check that $a^{w+2}_{5k+1}$ is not an integer for $k=1,2,\ldots,4223$, which can be done promptly with an easy computer calculation.

Similarly, we have 
{\small 
\begin{equation*}
\begin{split}
5^4a^{w+4}_{5k+2}&=24\cdot 211k+10546-\frac{250k^{4} + 1824400k^{3} + 2217840k^{2} + 868384k + 108736}{(5k+2)^4},\\ 
5^4a^{w+6}_{5k+3}&=24\cdot 211k+9519-\frac{375k^{4} + 1824900k^{3} + 3255210k^{2} + 1905444k + 366039}{(5k+3)^4},\\ 
5^4a^{w+8}_{5k+4}&=24\cdot 211k+10220-\frac{500k^{4} + 2113600k^{3} + 4849920k^{2} + 3697984k + 936320}{(5k+4)^4},\\ 
5^4a^{w+10}_{5k+4}&=24\cdot 211k+15236-\frac{500k^{4} + 1825600k^{3} + 4648320k^{2} + 3938464k + 1110416}{(5k+4)^4},\\ 
\end{split}\end{equation*}
}

\noindent
from which one only have to check the non-integrality of the respective values for $k$ up to 4863, 7295, 16895, 14591, respectively.

\section{Explicit examples}\label{sec-cases}
In the sequel, we denote by ${\displaystyle \delta=q\frac{d}{dq}}$ the Euler differential.
We recall Ramanujan's classical results
\begin{equation}\label{eqn-Ramanujan1}
\delta E_2=\frac{E^2_2-E_4}{12},\quad \delta E_4=\frac{E_2E_4-E_6}{3},\quad \delta E_6=\frac{E_2E_6-E_4^2}{2},
\end{equation}
and their consequence
\begin{equation}\label{eqn-Ramanujan2}
\delta^2 E_2=\frac{1}{72}(E^3_2-3E_2E_4+2E_6),\quad \delta^2 E_4=\frac{5}{36}(E^2_2E_4-2E_2E_6+E^2_4).
\end{equation}

\begin{prop}[weight 6 depth $3$]\label{prop-w6d3}
The $q$-expansion of the normalized extremal quasimodular form $f^{(3)}_6$ of weight $6$ and depth $3$ is given by 
\begin{equation}\label{eqn-w6d3}
f^{(3)}_6=\frac{5E^3_2-3E_4E_2-2E_6}{51840}=\frac{1}{6}\sum_{n=1}^{\infty}\big(n\sigma_3(n)-n^2\sigma_1(n)\big)q^n.
\end{equation}
It has integral $q$-expansion with all coefficients of $q^n$ for $n\geq 2$ being positive.
\end{prop}

\begin{proof}
The first equality of (\ref{eqn-w6d3}) is well-known (cf.\ \cite[p.459]{KK1}).
Then by (\ref{eqn-Ramanujan1}) and (\ref{eqn-Ramanujan2}), we calculate
$$
f^{(3)}_6=\frac{5E^3_2-3E_4E_2-2E_6}{51840}=\frac{\delta E_4}{1440}+\frac{\delta^2 E_2}{144}=\frac{1}{6}\sum_{n=1}^{\infty}\big(n\sigma_3(n)-n^2\sigma_1(n)\big)q^n,
$$
whence the second equality.
We then calculate $a(n)=n\sigma_3(n)-n^2\sigma_1(n)$ for $n\geq 2$ as
\begin{equation*}
\begin{split}
a(n)&=\sum_{d|n}(nd^3-n^2d)=\sum_{n=d_1d_2}nd^2_1(d_1-d_2)\\
&=\sum_{{n=d_1d_2}\atop d^2_1\leq n}n(d^2_1-d^2_2)(d_1-d_2)=\sum_{{n=d_1d_2}\atop d^2_1\leq n}d_1d_2(d_1-d_2)^2(d_1+d_2).
\end{split}
\end{equation*}
It then follows that each $a(n)$ is positive; moreover, it is easy to see that $a(n)$ is a multiple of $6$.
Hence $f^{(3)}_6$ has integral $q$-expansion with all coefficients of $q^n$ for $n\geq 2$ being positive, as desired.
\end{proof}

\begin{rem}\label{rem-w6d3}
It has been shown in \cite{D} that the $d$-th $q$-coefficient of $f^{(3)}_6$ ($=F_2(q)$ in [loc.\ cit., p.157]) counts the numbers of simply ramified coverings of genus two and degree $d$ of an elliptic curve over $\CC$, which already shows the integrality and positivity in Proposition \ref{prop-w6d3}.
\end{rem}

\begin{prop}[weight $8$ depth $2$]\label{prop-w8d2}
The $q$-expansion of the normalized extremal quasimodular form $f^{(2)}_8$ of weight $8$ and depth $2$ is given by 
\begin{equation}\label{eqn-w8d2}
f^{(2)}_8=\frac{5E^2_4+2E_6E_2-7E_4E_2^2}{362880}=\frac{1}{30}\sum_{n=1}^{\infty}\big(n\sigma_5(n)-n^2\sigma_3(n)\big)q^n.
\end{equation}
It has integral $q$-expansion with all coefficients of $q^n$ for $n\geq 2$ being positive.
\end{prop}

\begin{proof}
The first equality of (\ref{eqn-w8d2}) is well-known (cf.\ \cite[p.459]{KK1}).
By (\ref{eqn-Ramanujan1}) and (\ref{eqn-Ramanujan2}), we calculate
$$
f^{(2)}_8=\frac{5E^2_4+2E_6E_2-7E_4E_2^2}{362880}=-\frac{\delta E_6}{15120}-\frac{\delta^2 E_4}{7200}=\frac{1}{30}\sum_{n=1}^{\infty}\big(n\sigma_5(n)-n^2\sigma_3(n)\big)q^n,
$$
whence the second equality.
Similarly to the proof of Proposition \ref{prop-w6d3}, we calculate $a(n)=n\sigma_5(n)-n^2\sigma_3(n)$ for $n\geq 2$ as
$$
a(n)=\sum_{{n=d_1d_2}\atop d^2_1\leq n}d_1d_2(d_1-d_2)^2(d_1+d_2)(d^2_1+d^2_2).
$$
It then follows that each $a(n)$ is positive, and by a straightforward checking, one sees easily that each $d_1d_2(d_1-d_2)^2(d_1+d_2)(d^2_1+d^2_2)$ is divisible by the primes $2$, $3$, and $5$, hence is a multiple of $30$.
Hence $f^{(2)}_8$ has integral $q$-expansion with all coefficients of $q^n$ for $n\geq 2$ being positive, as desired.
\end{proof}

\subsection{End of the proof of Theorem \ref{thm-main}}\label{sub-proof}
It has been shown in \S\ref{sec-d1} that, in depth one $(r=1)$, the set $\mathcal{E}_1$ is contained in 
$$
\{2,6,8,10,12,14,16,18,20,22,24,28,30,32,34,38,54,58,68,80,114,118\}.
$$
In \S\ref{sec-d2}, we have seen that, in depth two ($r=2)$, only weights $4$ and $8$ are possible for the integrality of the $q$-expansion.
The normalized extremal quasimodular form $f^{(2)}_4$ of weight $4$ and depth $2$ is known to be 
$$
\frac{E_4-E^2_2}{288}=-\frac{\delta E_2}{24}
$$
(cf.\ \cite[Example 1.4]{KK1}), which is obviously of integral $q$-expansion.
We have seen in Proposition \ref{prop-w8d2} that the normalized extremal quasimodular form $f^{(2)}_8$ of weight $4$ and depth $2$ has integral $q$-expansion.
Thus we conclude that $\mathcal{E}_2=\{4,8\}$.

As for depth three $(r=3)$, in \S\ref{sec-d3}, we have shown that $\mathcal{E}_3$ is contained in the singleton set $\{6\}$.
It is well-known (\cite{D}) that the normalized extremal quasimodular form $f^{(3)}_6$ of weight $6$ and depth $3$ has integral $q$-expansion, which we have shown by an elementary argument in Proposition \ref{prop-w6d3}.
Hence we have $\mathcal{E}_3=\{6\}$.

Finally, in depth four $(r=4)$ case, we have shown in \S\ref{sec-d4} that there exists no weight $w$ that allows normalized extremal quasimodular form having integral $q$-expansion, i.e., $\mathcal{E}_4=\emptyset$.
This finishes the proof of Theorem \ref{thm-main}.

\frenchspacing
\begin{small}

\medskip\noindent
{\sc Department of Mathematics, Tokyo Institute of Technology, 2-12-1 Ookayama, Meguro, Tokyo 152-8551, Japan} (e-mail: {\tt mmts.o\_y@icloud.com})

\medskip\noindent
{\sc Department of Mathematics, Tokyo Institute of Technology, 2-12-1 Ookayama, Meguro, Tokyo 152-8551, Japan} (e-mail: {\tt bungen@math.titech.ac.jp})
\end{small}
\end{document}